\documentclass[11pt, a4paper,
oneside, headinclude,footinclude]{scrartcl}
\usepackage[utf8]{inputenc}
\usepackage{calc}

\usepackage[
nochapters, 
pdfspacing, 
dottedtoc 
]{classicthesis} 
\usepackage{amsopn}
\usepackage[T1]{fontenc} 
\usepackage{longtable}
\usepackage{multirow}
\usepackage[utf8]{inputenc} 
\usepackage{hhline}
\usepackage{graphicx} 
\graphicspath{{Figures/}} 
\usepackage{indentfirst}
\usepackage{enumitem} 

\usepackage{savesym}
\usepackage{mathrsfs}
\usepackage{xcolor}
\usepackage{geometry}
\usepackage[
    backend=biber,
    style=numeric,
    natbib=false,
    url=false, 
    doi=false,
    eprint=false,
    maxnames=50
]{biblatex}

\usepackage{esint}
\usepackage{subfig} 

\usepackage{amsmath,amssymb,amsthm,amsfonts} 

\usepackage{varioref} 

\usepackage{thmtools}
\usepackage{thm-restate}

\usepackage{hyperref}
\newcommand{\vertiii}[1]{{\left\vert\kern-0.25ex\left\vert\kern-0.25ex\left\vert #1 
    \right\vert\kern-0.25ex\right\vert\kern-0.25ex\right\vert}}
\usepackage{cleveref}
\theoremstyle{plain}
\newtheorem{teorema}{Theorem}[section]
\newtheorem{proposizione}[teorema]{Proposition}
\newtheorem{lemma}[teorema]{Lemma}
\newtheorem{corollario}[teorema]{Corollary}
\newtheorem*{theorem*}{Theorem}

\theoremstyle{definition}
\newtheorem{definizione}[teorema]{Definition}

\theoremstyle{remark}
\newtheorem{osservazione}[teorema]{Remark}

\renewcommand{\H}{\mathcal{H}}
\newcommand{\N}{\mathbb{N}}

\newcommand{\R}{\mathbb{R}}

\newcommand{\Tan}{\text{Tan}}

\newcommand{\eps}{\varepsilon}
\newcommand{\norm}{{\lVert\cdot\rVert}}

\DeclareMathOperator*{\supp}{supp}

\DeclareMathOperator*{\diam}{diam}

\DeclareMathOperator*{\dist}{dist}

\hypersetup{
colorlinks=true, breaklinks=true,bookmarksnumbered,
urlcolor=webbrown, linkcolor=RoyalBlue, citecolor=webgreen, 
pdftitle={}, 
pdfauthor={\textcopyright}, 
pdfsubject={}, 
pdfkeywords={}, 
pdfcreator={pdfLaTeX}, 
pdfproducer={LaTeX with hyperref and ClassicThesis} 
}

\title{\normalfont\spacedallcaps{Marstrand's density theorem for arbitrary norms in the plane}} 
\author{\spacedlowsmallcaps{Giacomo Del Nin\textsuperscript{\dag}, Andrea Merlo\textsuperscript{**}}}

\date{}

\begin{document}

\renewcommand{\sectionmark}[1]{\markright{\spacedlowsmallcaps{#1}}} 
\lehead{\mbox{\llap{\small\thepage\kern1em\color{halfgray} \vline}\color{halfgray}\hspace{0.5em}\rightmark\hfil}} 
\pagestyle{scrheadings}
\maketitle 
\setcounter{tocdepth}{2}

{\let\thefootnote\relax\footnotetext{\dag \textit{Max Planck Institute for Mathematics in the Sciences, Inselstrasse 22, 04103, Leipzig, Germany.}}}
{\let\thefootnote\relax\footnotetext{** \textit{Universidad del Pa\'is Vasco (UPV/EHU), Barrio Sarriena S/N 48940 Leioa, Spain.}}}

{\rightskip 1 cm
\leftskip 1 cm
\parindent 0 pt
\footnotesize

	%
{\textsc Abstract.} We show that if a non-trivial measure in the plane admits, at almost every point, positive and finite $\alpha$-dimensional density with respect to some norm, then $\alpha$ must be an integer.

\par
\medskip\noindent
{\textsc Keywords: Marstrand, measure, density, tangent measure, norm, barycenter.} 
\par
\medskip\noindent
{\textsc MSC (2020): 28A75 (primary); 28A78, 49Q15, 46B20 (secondary).} 
\par
}

\section{Introduction}
Since Besicovitch's foundational works, one of the key problems of Geometric Measure Theory has been to determine the geometric structure of measures with density, namely Radon measures $\mu$ on $\R^n$ such that the limit 
\begin{equation}
    \lim_{r\to 0}\frac{\mu(B(x,r))}{r^\alpha}
    \label{BP}
\end{equation}
exists and is positive and finite for some $\alpha>0$ and for $\mu$-almost every $x\in \R^n$. The study of this problem was initiated by Besicovitch in his works \cite{Bes1, Besicovitch2,Besicovitch3}, which could be considered the first contributions to Geometric Measure Theory. 

The above question, usually referred to as \emph{the density problem}, drove a lot of research during the 60s and 70s, see for instance \cite{Marstrand54, Marstrandoriginal, mattila75}, and was finally solved in 1987 in the celebrated work of D. Preiss, see \cite{Preiss1987GeometryDensities}. Preiss's solution of the density problem in \emph{Euclidean spaces} heavily relies on the fact that the Euclidean metric is induced by a scalar product. The extension to more general finite dimensional Banach spaces or even to metric spaces has been since a widely open question. More recently the second-named author has been able to solve the density problem in other metric spaces like the Heisenberg groups \cite{MerloG1cod,merloMM}, in the parabolic spaces with the collaboration of M. Mourgoglou and C. Puliatti \cite{merlo2022densityproblemparabolicspace} and in general Carnot groups in the case of sets with unit density in homogeneous groups, see \cite{julia_merlo_2023}. 

Despite all this progress, in each of these results the structure of the ball is of \emph{crucial} importance. Indeed, the first step to infer structural results for such measures is to show that their dimension is integral. For instance, in the case of homogeneous groups with polynomial norm this is quite easy to check following \cite{Kirchheim2002UniformilySpaces}, see also \cite{Chousionis2015MarstrandsGroup}. However, so far, with the exception of cases in which the dimension of the measure is known a priori to be $1$, see for instance \cite{preisstiserBesicovitch}, and where the dimension $1$ allows specific tricks independently on the metric, this paper represents the first contribution in the direction of studying the density problem for every bi-Lipschitz equivalent metric on a metric space.

Our main result is a generalization to all norms in the plane of the celebrated Marstrand's density theorem:

\begin{teorema}\label{thm:main}
    Let $\norm$ be an arbitrary norm in $\R^2$, and $\alpha\ge 0$. Let $\mu$ be a non-trivial measure in $\R^2$ such that
    \[
    \lim_{r\to0}\frac{\mu(B_\norm(x,r))}{r^\alpha}
    \]
    exists positive and finite for $\mu$-almost every $x$, where $B_\norm(x,r)$ is the ball with respect to the distance induced by the norm $\norm$. Then $\alpha$ must be an integer.
\end{teorema}

\subsection{The classical strategy}
A modern approach to the proof of Marstrand's density theorem is based on a contradiction argument (see e.g. \cite{DeLellis2008RectifiableMeasures}): suppose that there exists some non-trivial measure in $\R^n$ having non-integer $\alpha$-density almost everywhere, and consider the minimal ambient dimension $n$ where this happens. Then by taking suitable \textit{tangent measures} one can find a measure with the same property but supported on a hyperplane, thus contradicting the minimality.

The contradiction argument is based on two fundamental steps, which we are going to summarize below: a \textit{touching point argument} and the \textit{barycenter decay}.\\

\textit{Touching point argument.} Let $\mu$ be $\alpha$-uniform in $\R^n$, with $\alpha<n$. Then $\supp(\mu)$ has empty interior, hence for a point $x$ in its complement we can consider a maximal ball $B(x,r)$ in the complement, so that its boundary intersects $\supp(\mu)$ in at least one point $z$. By taking a tangent measure $\nu\in \Tan(\mu,z)$ we end up with an $\alpha$-uniform measure that in addition satisfies
\begin{equation}\label{eq:touching_point_intro}
\supp(\nu)\subseteq \{y\in\R^n:\, y\cdot e\ge 0\}\qquad\text{and}\qquad 0\in\supp(\nu),
\end{equation}
where $e=z-x$.\\

\textit{Barycenter decay.} Given an $\alpha$-uniform measure $\mu$, with $0\in\supp(\mu)$, let the barycenter function be defined by
\[
b_\mu(r):=\fint_{B(0,r)} zd\mu(z).
\]
The decay estimate consists in showing that 
\begin{equation}\label{eq:decay_intro}
|b_\mu(r)\cdot y|\le C |y|^2 \qquad\text{for every $y\in \supp(\mu)\cap B(0,r)$}.
\end{equation}
The power of this estimate is in the different scaling of the two sides: a priori we expect the quantity on the left-hand side to decay like $r|y|$, while instead the right-hand side decays like $|y|^2$. Considering very small $y$, or rather, fixing $y$ and zooming in around zero, this very strong information allows us to conclude that for every tangent measure $\nu\in\Tan_\alpha(\mu,0)$ it holds
\begin{equation}\label{eq:barycenter_zero_intro}
b_\nu(\rho)=0\quad\text{for every $\rho>0$}.
\end{equation}

Putting together \eqref{eq:barycenter_zero_intro} and \eqref{eq:touching_point_intro} we discover that $\supp(\nu)$ is actually contained in the hyperplane $\{y\in\R^n:\, y\cdot e=0\}$, and this concludes the argument.

We remark that estimate \eqref{eq:barycenter_zero_intro} is based on the polarization identity:
\[
2x\cdot y=|x|^2+|y|^2-|x-y|^2.
\]
Indeed, thanks to a shifting-ball argument one is able to conclude that 
\[
\left|\fint_{B(0,r)} |x|^2-|x-y|^2 d\mu\right|\le C|y|^2
\]
and this is enough to show \eqref{eq:decay_intro}.

\subsection{The updated strategy}
As mentioned above, estimate \eqref{eq:barycenter_zero_intro} relies crucially on the polarization identity, which famously characterizes inner product norms and hence is not directly applicable to other norms. However, the key observation is to view the polarization identity as the \textit{Taylor expansion} of the squared Euclidean norm $f(x)=|x|^2$ around the point $x$:
\[
\underbrace{|x-y|^2}_{f(x-y)}=\underbrace{|x|^2}_{f(x)}-\underbrace{2x\cdot y}_{\nabla f(x)\cdot(- y)}-\underbrace{|y|^2}_{o(|y|)}.
\]
The first modification is working with the polarization $V(x,y):=\tfrac12(\|x\|^2+\|y\|^2-\|x-y\|^2)$ as a replacement for the scalar product. Using a shifting-balls argument we show a quadratic decay for the polarization, which will be the replacement for \eqref{eq:decay_intro}:
\[
\left|\fint_{B(0,r)} V(z,y)d\mu(z)\right|\le C(\alpha)\|y\|^2\qquad\text{for every $y\in\supp(\mu)\cap B(0,r)$}.
\]
Secondly, using the Taylor expansion for any norm $\norm$, we discover that when $y$ is very small
\[
\fint_{B(0,r)} (\|x\|^2-\|x-y\|^2) d\mu(x) \approx \fint_{B(0,r)}\nabla\|x\|^2\cdot y \,d\mu(x).
\]
Observe that the points of non-differentiability of $\norm$ are always an $(n-1)$-rectifiable set (being the function convex), hence we can assume without loss of generality that $\norm$ is differentiable $\mu$-almost everywhere, and that $\mu$ admits such a Taylor expansion at $\mu$-almost every point. Indeed, otherwise we could take a tangent measure at a density point of the $(n-1)$-rectifiable set admitting a weak tangent, and find at least one tangent measure supported on a hyperplane, which would be a contradiction to the minimality of the ambient space, see Step (ii) in Section \ref{sec:proof}. 
We deduce that for any tangent measure $\nu\in \Tan(\mu,0)$ the \textit{nonlinear barycenters}
\[
b_\nu(\rho):=\fint_{B(0,\rho)} \nabla\|z\|^2\,d\nu(z)
\]
vanish for every $\rho>0$. This concludes the first step.

Observe now that, even assuming \eqref{eq:touching_point_intro} for some $e$, this might not be enough to conclude that $\supp(\nu)\subseteq\{y\in\R^n:\, y\cdot e=0\}$. Indeed, in the case of the Euclidean norm this conclusion is based on the implication
\begin{equation}\label{eq:monotonicity_intro}
y\cdot e>0 \implies \nabla\|y\|^2 \cdot e>0,
\end{equation}
which is not necessarily true for an arbitrary norm (actually, having this property for \textit{every} direction $e$ characterizes the Euclidean norm, see Remark \ref{rmk:ellipsoid}
). Nevertheless, we can show that every norm in the plane admits (up to a linear change of variables, which does not affect the problem) at least \textit{two} such good directions $v,w$ for which \eqref{eq:monotonicity_intro} holds. We would thus be able to conclude the argument if we could find touching points in either of these directions. In order to do this we run the touching point argument, but instead of using balls we use \textit{parallelograms} whose sides are orthogonal to $v$ and $w$ (observe that we can run the touching point argument with any given compact shape, there is no need for a relation with the balls of the norm $\norm$). In this way if we are able to find a touching point in the interior of one of the sides of the parallelogram we can reach the conclusion by taking a tangent measure, which will satisfy \eqref{eq:touching_point_intro} with $e$ equal to either $v$ or $w$, and using \eqref{eq:monotonicity_intro}. The only case that remains to analyze is if \textit{every} touching parallelogram touches the support of $\mu$ in one of the vertices. In this case, however, we are able to show that there must be a $1$-rectifiable piece in the support of $\mu$, hence we find anyway a tangent measure supported on a line, again a contradiction.

\subsection{Remarks}
We make some final remarks about the possibility of extending the argument to higher dimensions. The conclusion of the first step, namely finding an $\alpha$-uniform measure $\nu$ such that $b_\nu(\rho)=0$ for every $\rho>0$, works the same in any dimension. Obtaining the second step is more challenging, and in particular it is not clear whether, up to a linear change of variables, the implication \eqref{eq:monotonicity_intro} can hold for sufficiently many directions (e.g., for a spanning set of directions).  Even assuming this property however, there is one more case to consider: for example, in $\R^3$ if the touching parallelepipeds always touch the support of the measure along one of the 1-dimensional edges, an immediate adaptation of our planar argument does not imply that the support must contain a 2-rectifiable piece. Despite this, we suspect that a more careful adaptation should yield the needed $2$-rectifiable set in the support. Further work is thus needed to pass from the plane to higher dimensional Banach spaces. 


\subsection{Plan of the paper} In Section \ref{sec:prelims} we present some preliminaries about norms and their differentiability, uniform and tangent measures. In Section \ref{sec:barycenters} we introduce the non-linear barycenters, deduce the quadratic decay of the polarization average (see Lemma \ref{lemma:quadratic_decay}) and conclude the vanishing of the barycenters (Proposition \ref{prop:barycenter_zero} and Corollary \ref{cor:barycenter_zero}). In Section \ref{sec:touching_point} we prove some monotonicity properties of norms and run the touching point argument. Finally, in Section \ref{sec:proof} we prove Theorem \ref{thm:main}.

\section{Preliminaries}\label{sec:prelims}

\subsection{Cones}
Given $x\in\R^n$, a linear subspace $V$ and $M>0$ we define the \emph{bilateral} cone about $V$, centered at $x$, with aperture $M$ as
\begin{equation}
    \begin{split}
    X(x,V,M):=\{y:\, |\pi_{V^\perp}(y-x)|\le M |\pi_V(y-x)|\}.
    \nonumber
    \end{split}
\end{equation}
In addition, if $e\in \R^n$, we let $C(x,e,M)$ be the \emph{directional} cone centered at $x$ of axis $e$ aperture $M$ as
\begin{equation}
    \begin{split}
    C(x,e,M):&=\{y:\, |\pi_{e^\perp}(y-x)|\le M \langle e,y-x\rangle\}\\
    &= X(x,\mathrm{span}(e),M)\cap \{y:\langle e,y-x\rangle\geq 0\}.
    \nonumber
    \end{split}
\end{equation}
We also set $C(e,M):=C(0,e,M)$.

\subsection{Norms}
Throughout this paper, we denote by \(\lVert \cdot \rVert\) a norm on \(\mathbb{R}^n\) and by \(B_{\lVert \cdot \rVert}(x,r)\) the metric ball centered at \(x\) with radius \(r\), defined with respect to the distance induced by this norm. When the specific norm is clear from context, we will omit the explicit reference to \(\lVert \cdot \rVert\) in our notation, writing simply $B(x,r)$.

\begin{lemma}\label{lemma:non-differentiability}
    Let $\lVert\cdot\rVert$ be any norm in $\R^2$. Then the set $N$ of non-differentiability of $\lVert\cdot\rVert$ coincides with a (at most) countable union of lines through the origin.
\end{lemma}

\begin{proof}
Since $\lVert\cdot\rVert$ is a $1$-homogeneous function, $N$ must be dilation invariant.
Hence, $N$ can be written as a union of half lines that intersect only at the origin. Notice that this implies that $N\cap B(0,1)\setminus B(0,1/2)$ consists of a disjoint union of segments.
Thanks to \cite[Theorem 4.1]{zbMATH00219912}, we know that the set $N$ is countably $\mathcal{H}^1$-rectifiable. This implies that $\mathcal{H}^1\llcorner N$ is thus $\sigma$-finite and this in turn shows that there are only countably many segments in $N\cap B(0,1)\setminus B(0,1/2)$. This proves that the set $N$ is at most the disjoint union of countably many half lines through the origin. The fact that $\lVert \cdot\rVert$ is symmetric with respect to $0$ concludes the proof. 
\end{proof}

We recall for the reader's convenience the definition of the subdifferential of a convex function.

\begin{definizione}[Subdifferential]
    Let $f:\R^n\to \R$ be a convex function. We let $\partial f(x)$ be the \emph{subdifferential} of $f$ at $x$ the family of $v\in \R^n$ such that 
$$f(y)\geq f(x)+\langle v,y-x\rangle\qquad\text{ for every }y\in\R^n.$$
\end{definizione}

\begin{osservazione}
    Recall that the elements of the subdifferential satisfy the following inequality
\begin{equation}\label{eq:monotonicity_subdifferential}
    \langle v-w,x-y\rangle\ge 0\qquad\text{for every $x,y\in\R^n$, $v \in\partial f(x)$, $w \in\partial f(y)$.}
    \nonumber
\end{equation}    
\end{osservazione}

\begin{lemma}\label{lemma:quantitative_monotonicity}
    Let $\norm$ be a norm in $\R^n$. There exist $\delta>0$ and $\sigma>0$ such that for every $\nu\in\partial B(0,1)$
    \[
     \text{for every $x\in C(
     \nu,\sigma)\cap \partial B(0,1)$ and for every $v\in \partial (\lVert \cdot\rVert)(x)$ it holds }\quad v\cdot\nu \ge \delta.
    \]
\end{lemma}

\begin{proof}
Suppose by contradiction that this is not the case and that there exists a sequence of directions $\nu_i\in\partial B(0,1)$ for which there exists $x_i\in C(\nu_i,1/i)\cap\partial B(0,1)$ and there exists $v_i\in \partial (\lVert \cdot\rVert)(x_i)$ such that 
\begin{equation}
    \nu_i \cdot v_i\leq 1/i.
    \label{contradicip}
\end{equation}
By compactness of $\partial B(0,1)$ and the $1$-homogeneity of $\lVert \cdot\rVert$, we conclude that
$$\sup_{z\in \partial B(0,1)}\diam(\partial (\lVert \cdot\rVert)(z))<\infty\qquad \text{and}\qquad\inf_{z\in \partial B(0,1)} \mathrm{dist}(0,\partial (\lVert \cdot\rVert)(z))>0.$$
Since $\partial(\lVert \cdot\rVert)(x_i)$ are compact sets uniformly away from $0$ and contained in some bigger compact set, this shows that there exists $\nu\in \mathbb{S}^{1}$, $x\in \partial B(0,1)$ and a  $v\in \R^2\setminus \{0\}$ such that 
$$\nu_i\to \nu,\qquad x_i\to x= \nu\qquad\text{and}\qquad v_i\to v.$$
Let us check that $v\in \partial(\lVert \cdot\rVert)(x)$. Indeed, for every $q\in \R^2$ we have that
$$\lVert q\rVert\geq \lVert x_i\rVert+\langle v_i,q-x_i\rangle\qquad \text{for every }i\in\N.$$
Taking the limit as $i\to \infty$ this immediately shows that $0\neq v\in \partial(\lVert \cdot\rVert)(x)$ by the very definition of subdifferential. However \eqref{contradicip} shows that 
$\nu \cdot v\leq 0$. Notice that $B(0,1)\subseteq x+H^-_v$. Since $0\in B(0,1)$ this shows that $\langle v, \nu\rangle=0$ , indeed by definition of $H_v^-$
\[
0\le-\langle \nu,v\rangle=\langle 0-x,v\rangle\le 0
\]
Note that this shows that $B(0,1)\subseteq x+H^-_v=H^-_v$ as $x=\nu\perp v$. 
Thanks to the central symmetry of $B(0,1)$, we also infer that $-v\in \partial(\lVert\cdot\rVert)(-x)$ and thus $B(0,1)\subseteq H^-_{-v}=H^+_v$. This however proves that $B(0,1)\subseteq v^\perp$ which is not possible. This results in a contradiction and the result is proved.
\end{proof}

\subsection{Density, tangent and uniform measures}

\begin{definizione}
Let \(\alpha > 0\) be fixed. A Radon measure \(\phi\) on \(\mathbb{R}^n\) is said to have \((\alpha,\lVert\cdot\rVert)\)-\emph{density} if the limit
\[
\Theta^\alpha(\phi,x) := \lim_{r\to 0} \frac{\phi(B_{\lVert \cdot \rVert}(x,r))}{r^\alpha}
\]
exists and is both finite and nonzero for \(\phi\)-almost every \(x \in \mathbb{R}^n\).
\end{definizione}

\begin{osservazione}
In what follows, we denote by \(\mu_k \rightharpoonup \mu\) the weak-* convergence of the measures \(\mu_k\) to the measure \(\mu\).
\end{osservazione}

\begin{definizione}[Tangent measures]
Let \(\phi\) be a Radon measure on \(\mathbb{R}^n\) with \((\alpha, \lVert \cdot \rVert)\)-density. For any \(x \in \mathbb{R}^n\) and any \(r > 0\), define the measure \(T_{x,r}\phi\) by
\[
T_{x,r}\phi(A) = \phi(x + rA)
\]
for every Borel set \(A \subseteq \mathbb{R}^n\). The set of tangent measures to \(\phi\) at \(x\), denoted by \(\Tan_\alpha(\phi,x)\), consists of those Radon measures \(\mu\) for which there exists a sequence \(r_i \to 0\) such that
\[
r_i^{-\alpha} T_{x,r_i}\phi \rightharpoonup \mu.
\]
\end{definizione}

\begin{osservazione}
    It is well known that $\Tan_\alpha(\phi,x)$ is non empty for $\phi$-almost every $x\in \R^n$, see \cite[Proposition 3.4]{DeLellis2008RectifiableMeasures}.
\end{osservazione}

\begin{osservazione}\label{changeofvariables}
    Recall that given a Borel map $g:\R^n\to \R^n$, we let the push-forward of a Radon measure $\mu$ under $g$ be the measure $g_\#\mu$ that acts as 
    $$g_\#\mu(A)=\mu(g^{-1}(A))\qquad\text{for every Borel set }A.$$
    In the specific case when $g(\cdot)=x+ r_i\cdot$ we will always denote $g_\#\mu:=T_{x,r_i}\mu$. A simple computation further shows that the following change of variable formula holds
\[
\int f(z)d\mu(z)=\int f(y r_i)d (T_{0,r_i}\mu)(y),
\]
for every $\mu$-measurable function $f$.
\end{osservazione}

\begin{definizione}[Uniform measures]\label{uniform}
A Radon measure \(\mu\) is called an \((\alpha, \lVert \cdot \rVert)\)-uniform measure if \(0 \in \operatorname{supp}(\mu)\) and
\[
\mu(B(x,r)) = r^\alpha \quad \text{for every } r > 0 \text{ and every } x \in \operatorname{supp}(\mu).
\]
We denote the set of \((\alpha, \lVert \cdot \rVert)\)-uniform measures by \(\mathcal{U}_{\lVert \cdot \rVert}(\alpha)\). In what follows, we will often drop the explicit dependence on \(\lVert \cdot \rVert\) since the norm will usually be fixed from the context.
\end{definizione}

\begin{proposizione}\label{prop:tangents_are_uniform}
Let \(\phi\) be a measure with \((\alpha,\lVert \cdot\rVert)\)-density on \(\mathbb{R}^n\). Then, for \(\phi\)-almost every \(x \in \mathbb{R}^n\), it holds that
\[
\Tan_\alpha(\phi,x) \subseteq \Theta^\alpha(\phi,x)\,\mathcal{U}(\alpha).
\]
\end{proposizione}

\begin{proof}
See \cite[Proposition 3.4]{DeLellis2008RectifiableMeasures}. The proof there is written for the Euclidean norm, but the arguments work verbatim for any norm.
\end{proof}

\begin{proposizione}\label{preiss}
Let \(\phi\) be a Borel measure with \((\alpha,\lVert \cdot\rVert)\)-density in \(\mathbb{R}^n\). Then, for \(\phi\)-almost every \(x \in \mathbb{R}^n\) and every \(\mu \in \Tan_\alpha(\phi,x)\), we have
\[
r^{-\alpha}T_{y,r}\mu \in \Tan_\alpha(\phi,x) \quad \text{for every } y \in \operatorname{supp}(\mu) \text{ and every } r > 0.
\]
\end{proposizione}

\begin{proof}
For a proof in the context of general metric groups, see \cite[Proposition 3.1]{Mattila2005MeasuresGroups}.
\end{proof}

\begin{proposizione}\label{propspt1}
Let \(\phi\) be a Radon measure with \(\alpha\)-density and suppose that \(\mu \in \Tan_\alpha(\phi,x)\) is obtained as the limit $r_i^{-\alpha}T_{x,r_i}\phi \rightharpoonup \mu$
for some sequence \(r_i \to 0\). Then, for every \(y \in \operatorname{supp}(\mu)\), there exists a sequence \(\{z_i\}_{i\in\mathbb{N}} \subseteq \operatorname{supp}(\phi)\) such that
\[
\frac{z_i - x}{r_i} \to y.
\]
\end{proposizione}

\begin{proof}
See \cite[Proposition 3.4]{DeLellis2008RectifiableMeasures}.
\end{proof}





\begin{proposizione}\label{uniformup}
Assume $\mu$ is an $(\alpha,\lVert \cdot\rVert)$-uniform measure. Then for every $z\in\supp(\mu)$  we have
$$\emptyset\neq\Tan_\alpha(\mu, z)\subseteq \mathcal{U}(\alpha).$$
\end{proposizione}

\begin{proof}
A straightforward adaptation of the proof of \cite[Lemma 3.6]{DeLellis2008RectifiableMeasures} yields the desired conclusion.
\end{proof}

The following is a compactness result for uniform measures and for their supports.

\begin{lemma}\label{replica}
If $\{\mu_i\}_{i\in\N}$ is a sequence of $(\alpha,\lVert \cdot\rVert)$-uniform measures converging in the weak topology to some measure $\nu$, then
\begin{itemize}
\item[(i)] $\nu$ is an $(\alpha,\lVert \cdot\rVert)$-uniform measure,
\item[(ii)] if $y\in\supp(\nu)$ there exists a sequence $\{y_i\}\subseteq \R^n$ such that $y_i\in\supp(\mu_i)$ and $y_i\to y$,
\item[(iii)] if there exists a sequence $\{y_i\}\subseteq \R^n$ such that $y_i\in\supp(\mu_i)$ and $y_i\to y$, then $y\in\supp(\nu)$.
\end{itemize}
\end{lemma}

\begin{proof}
See \cite[Proposition 3.4]{DeLellis2008RectifiableMeasures}.
\end{proof}

\begin{osservazione}\label{rmk:not_everything}
    Let $\mu$ be a non trivial $(\alpha,\norm)$-uniform measure in $\R^n$. Then $\alpha\in [0,n]$. Moreover, if $\alpha<n$ then $\supp(\mu)$ has empty interior, and thus its complement is non empty. For a proof of this, the argument from \cite[Remark 3.14]{DeLellis2008RectifiableMeasures} can be adapted word by word by differentiating the measure with respect to the balls of the norm ${\norm}$ instead of the Euclidean ones.
\end{osservazione}

\begin{osservazione}[No $\alpha$-uniform measures for $0<\alpha<1$]\label{rmk:uniform<1}
    In the following we will use that there are no non-trivial $\alpha$-uniform measures when $\alpha\in(0,1)$. This fact actually holds in metric spaces, and we recall a proof here (see also \cite[Section~3.2]{Lorent-Polytope} and \cite[Lemma~30]{Marstrand54}). Let us suppose that $\mu$ is a non-trivial $\alpha$-uniform measure in a complete metric space $(X,d)$, with $\alpha>0$. Fix any point $x_0\in\supp(\mu)$. Fix an integer $N\ge2$, and for every $i\in\{1,\ldots,N\}$ let us define the annuli 
    \[
    A_i(x_0):=B\left(x_0,\frac{2i}{2N}\right)\setminus B\left(x_0,\frac{2i-1}{2N}\right).
    \]
    By uniformity (and the fact that $\alpha>0$) every such annulus has positive $\mu$-measure, and thus we can find a point $x_i\in A_i(x_0)\cap\supp(\mu)$. Moreover, by triangle inequality the balls $B(x_i,\frac{1}{2N})$ are pairwise disjoint and contained in $B(x_0,2)$. It follows that
    \[
    2^\alpha=\mu(B(x_0,1))\ge \sum_{i=1}^{N} \mu(B(x_0,\tfrac{1}{2N}))= N(2N)^{-\alpha}=2^{-\alpha}N^{1-\alpha}.
    \]
    By taking $N$ large enough we discover that $\alpha$ must be at least $1$.
\end{osservazione}

We conclude this section with a remark about the integral of $\norm$-radially symmetric functions with respect to $\norm$-uniform measures.

\begin{definizione}[Radially symmetric functions]
We say that a function $\varphi:\R^n\rightarrow\R$ is radially symmetric (with respect to a norm $\norm$) if there exists a profile function $g:[0,\infty)\rightarrow \R$ such that $\varphi(z)=g(\lVert z\rVert)$.
\end{definizione}

Integrals of radially symmetric functions with respect to uniform measures are easy to compute and we have the following change of variable formula.

\begin{proposizione}\label{prop5}
Let $\mu\in\mathcal{U}(m)$ and suppose $\varphi:\R^n\rightarrow\R$ is a radially symmetric non-negative function. Then, for any $x\in\supp(\mu)$ we have
\begin{equation}
\int  \varphi(z-x)d\mu(z)=m\int_0^\infty r^{m-1}g(r)dr\nonumber,
\end{equation}
where $g$ is the profile function associated to $\varphi$. In particular, the integral on the left-hand side is independent of $x\in\supp(\mu)$.
\end{proposizione}

\begin{proof}
First one proves the formula for simple functions of the form
$\varphi(z):=\sum_{i=1}^k a_i\chi_{B_{r_i}(0)}$,
where $a_i,r_i\geq 0$ for any $i=1,\ldots,k$. Then, one proves the result for a general $\varphi$ by Beppo Levi's convergence theorem. See also \cite[Remark 3.15]{DeLellis2008RectifiableMeasures}
\end{proof}

\section{Barycenters}\label{sec:barycenters}

Throughout this section we will suppose that $\lVert\cdot\rVert$ is a fixed norm on $\R^n$ and $\mu$ is an $\alpha$-uniform measure (with respect to $\norm$). 
We assume further that the set
$$N:=\{z\in \R^n:\lVert \cdot\rVert^2 \text{ is not differentiable at } z\}$$
is $\mu$-null, so that quantities involving the integral of $\nabla \norm^2$ in $d\mu$ are well defined. This will not be an issue since, in the case that $\mu$ is not differentiable on a positive measure set, we immediately find a contradiction, see Step (ii) in Section \ref{sec:proof}.

\begin{definizione}[Barycenter]
For every $x\in \supp(\mu)$ and $r>0$ we define the following objects
\begin{itemize}
    \item[(i)] the \emph{non-linear barycenter } 
\begin{equation}\label{eq:nonlinear_barycenter}
b_\mu(r):=\frac{1}{r^\alpha}\int_{B(0,r)} \nabla\lVert\cdot\rVert^2(z)d\mu(z)\quad\in\R^n;
\end{equation}
    \item[(ii)] the \emph{polarization average in direction $y$}
    \[
    b_\mu(r;y):=\frac{1}{r^\alpha} \int_{B(0,r)} V(z,y)d\mu(z)\quad\in\R,
    \]
    where $y\in \R^n$ and
    \begin{equation}\label{eq:polarization}
    V(z,y):=\frac{\lVert z\rVert^2+\lVert y\rVert^2-\lVert z-y\rVert^2}{2}.
    \end{equation}
\end{itemize}
\end{definizione}
The quantity $b_\mu(r;y)$ plays the role of $b_\mu(r)\cdot y$, at least in the approximation when $y$ is very small compared to $r$.

\begin{osservazione}\label{rmk:triangle_inequality}
    Observe that it still holds
\[
| b_\mu(r;y)|\le r\|y\|.
\]
This is a consequence of the Cauchy-Schwarz inequality for $V$
$$V(z,u)\leq \lVert z\rVert\lVert u\rVert.$$
Such inequality can be easily obtained by the triangle inequality as shown here below. First
\begin{equation}
    \begin{split}
        2V(z,u)=\lVert z\rVert^2+\lVert u\rVert^2-\lVert z-u\rVert^2 \leq \lVert z\rVert^2+\lVert u\rVert^2-\lVert z\rVert^2-\lVert u\rVert^2+2\lVert z\rVert\lVert u\rVert
        \leq 2\lVert z\rVert\lVert u\rVert,
        \nonumber
    \end{split}
\end{equation}
and similarly
\begin{equation}
    \begin{split}
        2V(z,u)=\lVert z\rVert^2+\lVert u\rVert^2-\lVert z-u\rVert^2 \geq \lVert z\rVert^2+\lVert u\rVert^2-\lVert z\rVert^2-\lVert u\rVert^2-2\lVert z\rVert\lVert u\rVert
        \geq -2\lVert z\rVert\lVert u\rVert.
        \nonumber
    \end{split}
\end{equation}
This concludes the proof.

\end{osservazione}

\subsection{Quadratic decay}
One of the cornerstones of Martsrand's argument is the quadratic decay of the barycenters. In the next Lemma we show that the same decay for the polarization $V$ holds for any norm in $\R^n$. We adapt the original strategy, which is based on a shifting-balls argument (see Figure \ref{fig:shifting_balls}).

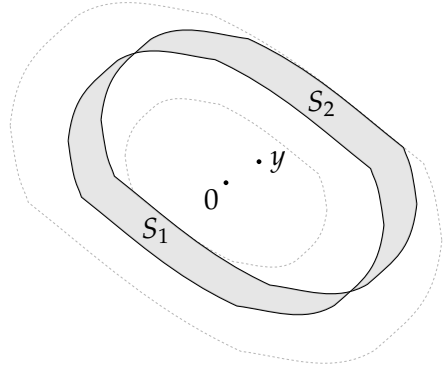
\begin{figure}
    \centering
    \def\svgscale{0.4}{
\begingroup%
  \makeatletter%
  \providecommand\color[2][]{%
    \errmessage{(Inkscape) Color is used for the text in Inkscape, but the package 'color.sty' is not loaded}%
    \renewcommand\color[2][]{}%
  }%
  \providecommand\transparent[1]{%
    \errmessage{(Inkscape) Transparency is used (non-zero) for the text in Inkscape, but the package 'transparent.sty' is not loaded}%
    \renewcommand\transparent[1]{}%
  }%
  \providecommand\rotatebox[2]{#2}%
  \newcommand*\fsize{\dimexpr\f@size pt\relax}%
  \newcommand*\lineheight[1]{\fontsize{\fsize}{#1\fsize}\selectfont}%
  \ifx\svgwidth\undefined%
    \setlength{\unitlength}{403.27298382bp}%
    \ifx\svgscale\undefined%
      \relax%
    \else%
      \setlength{\unitlength}{\unitlength * \real{\svgscale}}%
    \fi%
  \else%
    \setlength{\unitlength}{\svgwidth}%
  \fi%
  \global\let\svgwidth\undefined%
  \global\let\svgscale\undefined%
  \makeatother%
  \begin{picture}(1,0.84222995)%
    \lineheight{1}%
    \setlength\tabcolsep{0pt}%
    \put(0,0){\includegraphics[width=\unitlength,page=1]{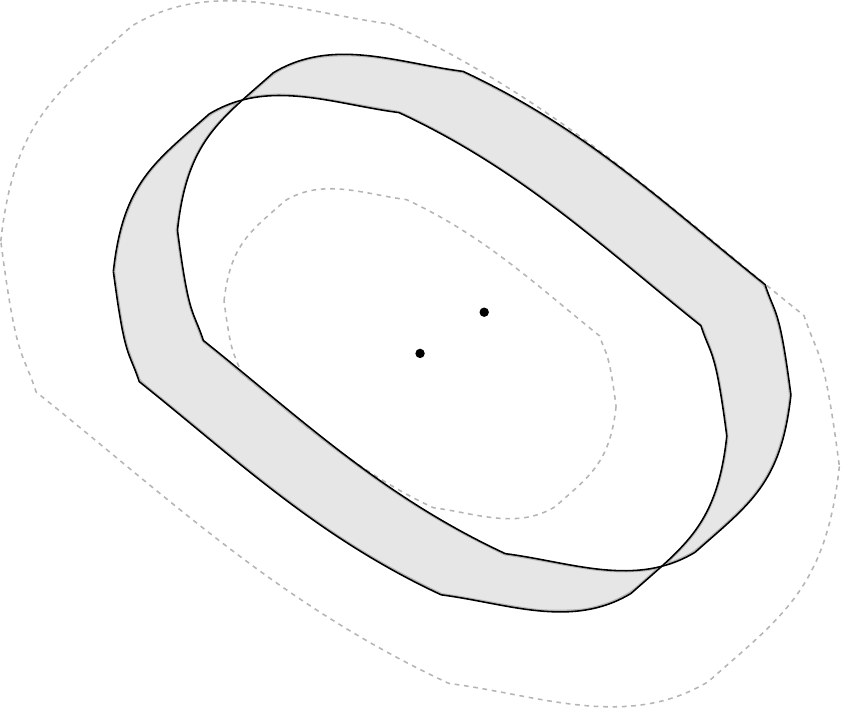}}%
    \put(0.60293446,0.46328572){\makebox(0,0)[lt]{\lineheight{1.25}\smash{\begin{tabular}[t]{l}$y$\end{tabular}}}}%
    \put(0.44917439,0.36118767){\makebox(0,0)[lt]{\lineheight{1.25}\smash{\begin{tabular}[t]{l}$0$\end{tabular}}}}%
    \put(0.30472815,0.29206042){\makebox(0,0)[lt]{\lineheight{1.25}\smash{\begin{tabular}[t]{l}$S_1$\end{tabular}}}}%
    \put(0.68793171,0.58758384){\makebox(0,0)[lt]{\lineheight{1.25}\smash{\begin{tabular}[t]{l}$S_2$\end{tabular}}}}%
  \end{picture}%
\endgroup%
}
    \caption{The shifting balls argument. The (boundaries of the) balls $B(0,r)$ and $B(y,r)$ are depicted with a black line. The (boundaries of the) balls $B(0,r+\|y\|)$ and $B(0,r-\|y\|)$ are depicted with a dashed gray line.}
    \label{fig:shifting_balls}
\end{figure}

\begin{lemma}[Quadratic decay of polarization]\label{lemma:quadratic_decay} Let $\mu$ be an $\alpha$-uniform measure. Then for every $y\in B(x,r)\cap \supp(\mu)$ it holds that $| b_\mu(r;y)|\le C(\alpha)\|y\|^2$.
\end{lemma}

\begin{proof} 
    First observe that if $y\in B(0,r)\setminus B(0,r/2)$ then the inequality holds true because then $| b_\mu(r;y)|\le r\|y\|\le 2\|y\|^2$.
    We can thus suppose in the following that $y\in B(0,r/2)$.
    We have
    \[
    2| b_\mu(r;y)|\le \frac{1}{r^\alpha}\int_{B(0,r)} \|y\|\,^2\, d\mu(z)+\frac{1}{r^\alpha} \left|\int_{B(0,r)} \|z\|\,^2\, d\mu(z)-\int_{B(0,r)} \|z-y\|\,^2\, d\mu(z)\right|.
    \]
    The first term coincides with $\|y\|^2$, so we just need to estimate the second. Set $S_1:=B(0,r)\setminus B(y,r)$ and $S_2:=B(y,r)\setminus B(0,r)$. Notice that since the measure $\mu$ is $\alpha$-uniform, we have  
    \[
    \mu(S_1)=\mu(B(0,r))-\mu\big(B(0,r)\cap B(y,r))=\mu(B(y,r))-\mu\big(B(0,r)\cap B(y,r))=\mu(S_2).
    \]
    Thus, since $y$ and $0$ are both in $\supp(\mu)$, we can rewrite the last term as
    \begin{align}
    \frac{1}{r^\alpha} \left|\int_{B(0,r)} \|z-y\|\,^2\, d\mu(z)-\int_{B(y,r)} \|z-y\|\,^2\, d\mu(z)\right|&=\frac{1}{r^\alpha} \left|\int_{S_1} \|z-y\|\,^2\, d\mu(z)-\int_{S_2} \|z-y\|\,^2\, d\mu(z)\right|\nonumber\\
    & \le \frac{1}{r^\alpha}\left( \mu(S_2) (r+\|y\|)^2-\mu(S_1) (r-\|y\|)^2\right), \label{eq:moon_difference}
    \end{align}
where the last inequality comes from the fact that 
$$r\le\lVert z-y\rVert\leq r+\|y\| \quad\text{on $S_2$}\qquad \text{and}\qquad  r-\|y\|\le \lVert z-y\rVert\leq r\quad\text{on $S_1$}.$$

    Moreover,
    \begin{align*}
        \mu(S_1)\le \mu(B(0,r))-\mu(B_{r-\|y\|}(0))&=(r^\alpha-(r-\|y\|)^\alpha)=r^\alpha\left(1-\left(1-\frac{\|y\|}{r}\right)^\alpha\right)\\
        & \le r^\alpha C\alpha \frac{\|y\|}{r}=C(\alpha)r^{\alpha-1}\|y\|.
    \end{align*}
    It follows that
    \begin{align*}
     \frac{1}{r^\alpha} \left(\mu(S_2) (r+\|y\|)^2-\mu(S_1) (r-\|y\|)^2\right)&=\frac{1}{r^\alpha}\mu(S_1)\left((r+\|y\|)^2-(r-\|y\|)^2\right)\\
     &=\frac{1}{r^\alpha}\mu(S_1)4r\|y\|\\
     & \le \frac{1}{r^\alpha} C(\alpha)r^{\alpha-1} \|y\|4r\|y\|\\
     &= 4C(\alpha)\|y\|^2.\qedhere
    \end{align*}
\end{proof}

\subsection{Vanishing barycenters}
From the previous lemma we know that $\fint_{B(0,r)} (\|z\|^2-\|z-y\|^2)d\mu(z)$ decays quadratically in $y$, for $y\in \supp(
\mu)$. On the other hand by the Taylor expansion of $\|\!\cdot\!\|^2$, and formally neglecting error terms, the integrand $\|z\|^2-\|z-y\|^2$ is approximately $\nabla\norm^2(z)\cdot y$, at least when $y$ is small compared to $z$. This suggests that for a tangent measure $\nu$ at $0$ we should have 
\[
\fint_{B(0,r)} \nabla\|z\|^2\cdot y \,d\mu(z)=0
\]
at least for $y\in\supp(u)$. To make this rigorous, let us introduce the remainders
\begin{align*}
\begin{aligned}
    \Delta(z,y):&=\|z-y\|^2-\|z\|^2+\nabla\norm^2(z)\cdot y\\
    &= \|y\|^2-2V(z,y)+\nabla\norm^2(z)\cdot y.
\end{aligned}
\end{align*}
Observe the $1$-homogeneity property 
\[
\nabla\norm^2(\lambda z)=\lambda \nabla \norm^2(z)\qquad\text{for every $z\in \R^n$ and every $\lambda\ge 0$}
\]
and the simultaneous $2$-homogeneity property of $V$
\begin{equation}\label{eq:2-homogeneity}
    V(\lambda z,\lambda y)=\lambda^2V(z,y)\qquad\text{for every $z,y\in\R^n$ and every $\lambda\ge 0$.}
\end{equation}
With the same computations as in Remark \ref{rmk:triangle_inequality} we deduce that 
\[
|\Delta(z,y)|\lesssim \|z\|\|y\|+\|y\|^2.
\]
Moreover, if $\norm$ is differentiable at $z$ then 
\[
\lim_{y\to 0}\frac{\Delta(z,y)}{\lVert y\rVert}=0,
\]
and thus by dominated convergence, for every fixed $r$ we have
\begin{equation}\label{eq:dominated}
\lim_{y\to0} \frac{1}{\|y\|} \int_{B(0,r)} |\Delta(z,y)|d\mu(z)=0.
\end{equation}
From
\[
2V(z,y)=\nabla(\lVert\cdot\rVert^2)(z)\cdot y+\lVert y\rVert^2-\Delta(z,y)
\]
we also obtain that
\begin{equation}\label{eq:bound}
    \left|\frac{1}{r^\alpha} \int_{B(0,r)}\nabla\lVert\cdot\rVert^2(z)d\mu(z)\cdot y\right|\le\left|\frac{1}{r^\alpha} \int_{B(0,r)}2V(z,y)d\mu(z)\right|+\|y\|^2 +\left|\frac{1}{r^\alpha} \int_{B(0,r)} \Delta(z,y)d\mu(z)\right|.
\end{equation}

\begin{proposizione}\label{prop:barycenter_zero}
    Let $\nu\in \Tan_\alpha(\mu,0)$. For every $w\in \supp(\nu)$ and every $\rho>0$ we have
    \[
    b_\nu(\rho)\cdot w=0.
    \]
\end{proposizione}

\begin{proof}
    Take $r_i\to 0$ such that $r_i^{-\alpha}T_{0,r_i} \mu\rightharpoonup \nu$. Given $w\in\supp(\nu)$ we can find $w_i:=y_i/r_i\to w$, with $y_i\in\supp(\mu)$. Select a subsequence $k(i)$ such that 
    \begin{equation}\label{eq:small_r_k(i)}
    \|y_{k(i)}\|/r_i\to 0.
    \end{equation}
    We apply \eqref{eq:bound} with $r=\rho r_i$ and $y=y_{k(i)}$, together with Lemma \ref{lemma:quadratic_decay}, to obtain
    \begin{align}
    \begin{aligned}\label{eq:Delta_bound}
        \left|\frac{1}{(\rho r_i)^\alpha} \int_{B(0,\rho r_i)}\nabla\lVert\cdot\rVert^2(z)d\mu(z)\cdot y_{k(i)}\right|\le & C(\alpha)\|y_{k(i)}\|^2+\|y_{k(i)}\|^2\\
        &+\left|\frac{1}{(\rho r_i)^\alpha} \int_{B(0,\rho r_i)} \Delta(z,y_{k(i)})d\mu(z)\right|.
    \end{aligned}
    \end{align} 
    We rewrite the last term by a change of variables, see Remark \ref{changeofvariables}, as
    \begin{align*}
        \frac{1}{(\rho r_i)^\alpha} \int_{B(0,\rho r_i)} \Delta(z,y_{k(i)})d\mu(z)&=\frac{1}{\rho ^\alpha} \int_{B(0,\rho )} \Delta(z'r_i,y_{k(i)})d(r_i^{-\alpha}T_{0,r_i}\mu)(z')\\
        &=r_i^2\frac{1}{\rho ^\alpha} \int_{B(0,\rho )} \Delta\left(z',\frac{y_{k(i)}}{r_i}\right)d(r_i^{-\alpha}T_{0,r_i}\mu)(z')
    \end{align*}
    where we used the 2-homogeneity of $V(z,y)$ (see \eqref{eq:2-homogeneity}). We also rewrite the first term as
    \begin{align*}
    \frac{1}{(\rho r_i)^\alpha} \int_{B(0,\rho r_i)}\nabla\lVert\cdot\rVert^2(z)d\mu(z)\cdot y_{k(i)}&=\frac{1}{\rho^\alpha} \int_{B(0,\rho )}\nabla\lVert\cdot\rVert^2(zr_i)d(r_i^{-\alpha}T_{0,r_i}\mu)(z)\cdot y_{k(i)}\\
    &=r_i\frac{1}{\rho^\alpha} \int_{B(0,\rho )}\nabla\lVert\cdot\rVert^2(z)d(r_i^{-\alpha}T_{0,r_i}\mu)(z)\cdot y_{k(i)}
    \end{align*}
    where in the last line we used the 1-homogeneity of $\nabla \lVert\cdot\rVert^2$.
    Dividing the last chain of equalities by $r_i r_{k(i)}$ and applying \eqref{eq:Delta_bound} we obtain that
    \begin{align}
    \begin{aligned}\label{eq:blowup_bound}
        \left|\frac{1}{\rho^\alpha} \int_{B(0,\rho )}\nabla\lVert\cdot\rVert^2(z)d(r_i^{-\alpha}T_{0,r_i}\mu)(z)\cdot \frac{y_{k(i)}}{r_{k(i)}}\right|\le &  (C+1)\frac{1}{r_i r_{k(i)}} \|y_{k(i)}\|^2\\
        &+\frac{1}{r_i r_{k(i)}} r_i^2\frac{1}{\rho ^\alpha} \int_{B(0,\rho )} \Delta\left(z',\frac{y_{k(i)}}{r_i}\right)d(r_i^{-\alpha}T_{0,r_i}\mu)(z').
    \end{aligned}
    \end{align}
    We claim that the right-hand side goes to zero as $i\to \infty$. The first term is going to zero by \eqref{eq:small_r_k(i)}. Since $y_{k(i)}/r_{k(i)}\to w$, for $i$ large enough we have $\|y_{k(i)}\|\ge \tfrac12\|w\| r_{k(i)}$. The second term can thus be estimated (for $i$ large enough) by
    \[
    \frac{r_i}{r_{k(i)}}\frac{1}{\rho ^\alpha} \int_{B(0,\rho )} \Delta\left(z',\frac{y_{k(i)}}{r_i}\right)d(\nu_i)(z')\le \frac{2}{\|w\|}\frac{1}{\rho ^\alpha} \int_{B(0,\rho )} \frac{r_i}{\|y_{k(i)}\|}\Delta\left(z',\frac{y_{k(i)}}{r_i}\right)d(\nu_i)(z'),
    \]
    where $\nu_i:=r_i^{-\alpha}T_{0,r_i}\mu$. 
    We now need to possibly select a further subsequence of $k(i)$. For fixed $z'$, we know that $\frac{1}{\|w\|}\Delta(z',w)\to 0$ as $w\to0 $, and thus by dominated convergence (compare with \eqref{eq:dominated}) we know that for every fixed $i$
    \[
    \int_{B(0,\rho )} \frac{1}{\|w\|}\Delta(z',w)d(\nu_i)(z')\le \frac{1}{i}\qquad\text{for $\|w\|\le \delta(i)$ small enough.}
    \]
    We can thus select a diagonal subsequence $y_{k'(i)}$ of $y_{k(i)}$ such that $y_{k'(i)}/r_i\le \delta(i)$. With this choice we ensure that 
    \[
    \frac{1}{\rho ^\alpha} \int_{B(0,\rho )} \frac{r_i}{\|y_{k(i)}\|}\Delta\left(z',\frac{y_{k(i)}}{r_i}\right)d(\nu_i)(z')\to 0 \qquad\text{as $i\to\infty$}.
    \]
    In conclusion, the right-hand side of \eqref{eq:blowup_bound} goes to zero, while the left-hand side converges to $b_\nu(\rho)\cdot w$. This concludes the proof.
\end{proof}

The following corollary now follows immediately.
    
    \begin{corollario}[Vanishing barycenter]\label{cor:barycenter_zero}
        Let $\nu\in \Tan_\alpha(\mu,0)$. Suppose further that $\supp(\nu)$ spans $\R^n$. Then $b_\nu(\rho)=0$ for every $\rho>0$.
    \end{corollario}

    \begin{osservazione}
        It would be interesting to understand whether the conclusion $b_\nu(\rho)=0$ for every $\rho>0$ holds also without assuming that $\supp(\nu)$ spans $\R^n$. In the Euclidean case this is true, since the barycenter is obviously contained in the span of the support, but this is not necessarily the case for a general norm.
    \end{osservazione}


\section{Touching point argument}\label{sec:touching_point}

Let us denote by $\mathrm{reg}(\partial K)$ the set of points in $\partial K$ where $\lVert\cdot\rVert$ is differentiable, and for such points let us denote by $n(x):=\frac{\nabla \|x\|}{|\nabla \|x\||}$ the normal to $\partial K$ at $x$. 

We now introduce an essential notion, that we call (strict) monotonicity. This is related to finding half spaces where $\nabla \lVert\cdot\rVert^2$ always points towards the "inside" of the half space.

\begin{definizione}[Direction of monotonicity]
    Let $K$ be an origin-symmetric convex body in $\R^n$. We say that $\nu\in\mathbb{S}^{n-1}$ is a direction of \textit{(weak) monotonicity} for $K$ if for every $x\in \mathrm{reg}(\partial K)$ we have that
    \[
    x\cdot\nu>0\implies n(x)\cdot \nu \ge0.
    \]
    We say that $\nu$ is a direction of \textit{strict monotonicity} if the inequality holds with the strict sign.
    If $\norm$ is a norm in $\R^n$, we say that $\nu$ is a direction of (strict) monotonicity for $\norm$ if it is so for its unit ball $K=B_{\norm}(0,1)$.
\end{definizione}

The way in which we will use the notion of strict monotonicity is related to the following rigidity property.

\begin{lemma}[Monotonicity and flatness]\label{lemma:monotonicity_and_flatness}
    Consider a norm $\lVert\cdot\rVert$ in $\R^n$. Assume that $\mu$ is an $\alpha$-uniform measure in $\R^n$ for which 
    \[
    0\in\supp(\mu),\qquad\supp(\mu)\subseteq \{x_n\ge 0\},
    \]
    and such that $b_\mu(\rho)=0$ for every $\rho>0$. Then the following hold:
    \begin{enumerate}
        \item If $e_n$ is a direction of strict monotonicity for $\lVert\cdot\rVert$, then $\supp(\mu)\subseteq \{x_n=0\}$.
        \item If $e_n$ is a direction of weak monotonicity for $\lVert\cdot\rVert$, then, setting $V:=\{x_n=0\}$, there exists $M\ge 0$ such that $\supp(\mu)\subseteq X(0,V,M)$.
    \end{enumerate}
\end{lemma}

\begin{proof} 
1. For every $\rho> 0$ we can write
\begin{align*}
    0=e_n\cdot b_\mu(\rho)&=\frac{1}{\rho^\alpha}\int_{B(0,\rho)} e_n\cdot\nabla(\lVert\cdot\rVert^2) (z) d\mu(z)\\
    &=\frac{1}{\rho^\alpha}\int_{B(0,\rho)\cap\{x_n=0\}} e_n\cdot\nabla(\lVert\cdot\rVert^2) (z) d\mu(z)+\frac{1}{\rho^\alpha}\int_{B(0,\rho)\cap \{x_n>0\}} e_n\cdot\nabla(\lVert\cdot\rVert^2) (z) d\mu(z)\\
    &\ge \frac{1}{\rho^\alpha}\int_{B(0,\rho)\cap \{x_n>0\}} e_n\cdot\nabla(\lVert\cdot\rVert^2) (z) d\mu(z)
\end{align*}
    and the last integrand is strictly positive by the strict monotonicity, so the only way for the integral to be zero is that $\mu(B(0,\rho)\cap \{x_n>0\})=0$ for every $\rho>0$. It follows that $\supp(\mu)\subseteq \{x_n=0\}$.

    2. We perform the same computation as above, but we split the integral over $B(0,\rho)$ as the integral over the two sets
    \[
    \{z\in B(0,\rho):\,n(z)\cdot e_n=0\}\qquad\text{and}\qquad \{z\in B(0,\rho):\,n(z)\cdot e_n>0\}.
    \]
    By Lemma \ref{lemma:quantitative_monotonicity} there exists $M$ such that the first set is contained in $X(0,V,M)$. This concludes the proof as in Point 1.
\end{proof}

\begin{figure}
    \centering
    \def\svgscale{0.7}{
\begingroup%
  \makeatletter%
  \providecommand\color[2][]{%
    \errmessage{(Inkscape) Color is used for the text in Inkscape, but the package 'color.sty' is not loaded}%
    \renewcommand\color[2][]{}%
  }%
  \providecommand\transparent[1]{%
    \errmessage{(Inkscape) Transparency is used (non-zero) for the text in Inkscape, but the package 'transparent.sty' is not loaded}%
    \renewcommand\transparent[1]{}%
  }%
  \providecommand\rotatebox[2]{#2}%
  \newcommand*\fsize{\dimexpr\f@size pt\relax}%
  \newcommand*\lineheight[1]{\fontsize{\fsize}{#1\fsize}\selectfont}%
  \ifx\svgwidth\undefined%
    \setlength{\unitlength}{483.69437631bp}%
    \ifx\svgscale\undefined%
      \relax%
    \else%
      \setlength{\unitlength}{\unitlength * \real{\svgscale}}%
    \fi%
  \else%
    \setlength{\unitlength}{\svgwidth}%
  \fi%
  \global\let\svgwidth\undefined%
  \global\let\svgscale\undefined%
  \makeatother%
  \begin{picture}(1,0.26204958)%
    \lineheight{1}%
    \setlength\tabcolsep{0pt}%
    \put(0,0){\includegraphics[width=\unitlength,page=1]{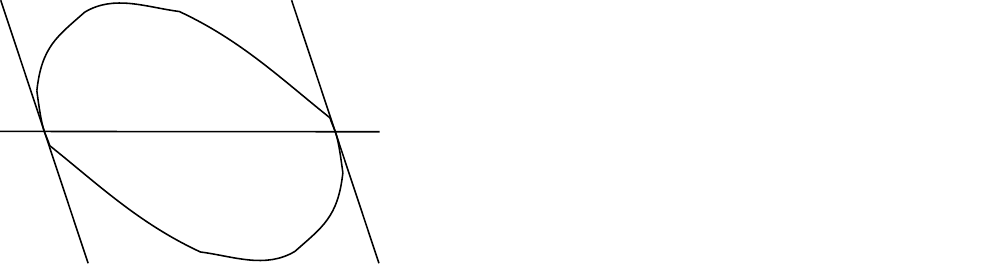}}%
    \put(0.21826878,0.05203256){\makebox(0,0)[lt]{\lineheight{1.25}\smash{\begin{tabular}[t]{l}$K$\end{tabular}}}}%
    \put(0.75458339,0.05203256){\makebox(0,0)[lt]{\lineheight{1.25}\smash{\begin{tabular}[t]{l}$AK$\end{tabular}}}}%
    \put(0,0){\includegraphics[width=\unitlength,page=2]{Monotonicity.pdf}}%
    \put(0.06244925,0.14075415){\makebox(0,0)[lt]{\lineheight{1.25}\smash{\begin{tabular}[t]{l}$p_-$\end{tabular}}}}%
    \put(0.34839103,0.14075415){\makebox(0,0)[lt]{\lineheight{1.25}\smash{\begin{tabular}[t]{l}$p_+$\end{tabular}}}}%
    \put(0.62411727,0.14075415){\makebox(0,0)[lt]{\lineheight{1.25}\smash{\begin{tabular}[t]{l}$p_-$\end{tabular}}}}%
    \put(0.91005898,0.14075415){\makebox(0,0)[lt]{\lineheight{1.25}\smash{\begin{tabular}[t]{l}$p_+$\end{tabular}}}}%
  \end{picture}%
\endgroup%
}
    \caption{Reference figure for the proof of Lemma \ref{lemma:monotonicity_linear}. The geometric idea is the following: the two supporting lines at the points $p_\pm$ can be made vertical by applying a shearing transformation that keeps the horizontal line fixed. This makes the vertical direction a direction of weak monotonicity.}
    \label{fig:monotonicity}
\end{figure}

The following lemma will be crucial in the proof of Marstrand's theorem. 

\begin{lemma}[Directions of strict monotonicity]\label{lemma:directions_of_monotonicity}
    Let $\lVert\cdot\rVert$ be a norm in $\R^2$, and let $K$ be the associated unit ball. Then 
    there exists a linear transformation $A:\R^2\to\R^2$ such that $AK$ has two independent directions of strict monotonicity.
\end{lemma}

\begin{proof}
    Let us consider a point $\bar{x}\in\partial K$ at maximum distance from the origin. Then $K$ is contained in the Euclidean ball $B(0,|\bar x|)$ and it follows that  $\bar{x}^\perp$ is a direction of strict monotonicity, so we only need to find a second one. To this aim, let us squeeze $K$ in direction $\bar{x}^\perp$. More precisely, assume for simplicity that $\bar{x}=e_1$, and let us apply the transformation 
    \[
    A_\eps(x_1,x_2)=(x_1,\eps x_2).
    \]
    We consider a point $x_\eps\in \partial (A_\eps K)$ at maximum distance from the origin. For $\eps>0$ small enough, $x_\eps$ will be distinct from $\pm A_\eps \bar x$. It follows that $x_\eps^\perp$ is also a direction of strict monotonicity, linearly independent from $A_\eps \bar x$. The set $A_\eps K$ has thus two linearly independent directions of strict monotonicity.
\end{proof}

\begin{lemma}[Directions of weak monotonicity]\label{lemma:monotonicity_linear}
    Let $\norm$ be a norm in $\R^2$ and $\ell=e^\perp$ be a line through the origin. Then there exists an invertible linear transformation $A:\R^2\to \R^2$ such that $A(\ell)=\ell$ and $e$ is a direction of weak monotonicity for the norm $\norm_A$.
\end{lemma}

\begin{proof}
The more geometrically minded might take Figure \ref{fig:monotonicity} as a full proof, but below we provide the complete argument.

Up to an orthogonal transformation we can assume that $e=e_2$. 
Following the notations of Lemma \ref{lemma:directions_of_monotonicity} we recall that we set $K=B(0,1)$ and that the norm $\lVert\cdot\rVert_A$ whose ball coincides with $AK$ is defined as $\lVert x\rVert_A:=\lVert A^{-1}x\rVert$. Let $w\in \R^n\setminus\{0\}$ be such that  $\pm w\in \partial(\lVert \cdot\rVert)(p_\pm)$ where $\{p_+,p_-\}=\partial K\cap \ell$. Notice that there exists such a $w$ because of the symmetry of $\lVert \cdot\rVert$ with respect to $0$. 

Let us notice that the vectors $w$ and $e$ are linearly independent, as otherwise $K$ would be contained in the line $\ell$.
Thus we can define $A$ to be an invertible linear transformation such that 
$$A e_1=e_1\qquad\text{and}\qquad A (Jw)= e_2,\qquad\text{where}\qquad J:=\begin{pmatrix}
    0 & 1\\
    -1 & 0
\end{pmatrix}.$$

In this way $A(p_\pm)=p_\pm$ and we claim that $p_\pm+\mathrm{span}(e_2)$ are supporting affine lines to $AK$ at $p_\pm$ respectively. This is immediate from the fact that $p_\pm+\mathrm{span}(Jw)$ are by definition supporting lines to the compact set $K$ at $p_\pm$, and that the linear transformation $A$ preserves this property.
As a consequence, up to a scaling we can thus assume that $p_+=\lambda e_1$, $p_-=-\lambda e_1$, and that $e_1\in\partial\norm(p_+)$, $-e_1\in\partial\norm(p_-)$. Let us consider any $x$ in the upper half-plane, namely $x=\lambda_1 e_1+\lambda_2 e_2$, with $\lambda_2> 0$. We have two cases: $\lambda_1\ge 0$ and $\lambda_1\le 0$. 

Suppose first that $\lambda_1\ge0$. We use the monotonicity property of the subdifferential (see \eqref{eq:monotonicity_subdifferential}) of $\norm$ with respect to the points $x$ and $\lambda_1 e_1$ and its $0$-homogeneity to infer that
    \begin{align*}
        0& \le \langle \nabla \norm (x)-\nabla \norm(\lambda_1 e_1),x-\lambda_1 e_1\rangle \\
        &= \langle \nabla \norm (x)-e_1,\lambda_2 e_2\rangle\\
        &= \langle \nabla \norm (x),\lambda_2 e_2\rangle,
    \end{align*}
    from which we deduce that $\langle\nabla\norm(x),e_2\rangle\ge 0$ since $\lambda_2>0$.
If $\lambda_1\le 0$ the reasoning is similar. This concludes the proof.
\end{proof}

\begin{osservazione}\label{rmk:ellipsoid}
    We mention for reference the following result: if a norm $\norm$ in $\R^n$, $n\ge 3$, is such that for \textit{every} direction $e$ there exists a linear map $A$ that fixes $e^\perp$ and such that $e$ is a direction of weak monotonicity for $\norm_A$, then $\norm$ is an inner product norm, and the unit ball is an ellipsoid. This follows from \cite[Theorem~10.2.3]{Schneider}. This clearly fails if $n=2$, since as we just showed \textit{every} norm in $\R^2$ has this property. 
\end{osservazione}

 \begin{definizione}[Parallelograms]
    Given two independent vectors $v,w\in\R^2$ we denote by $P_{v,w}$ (also just $P$ if $v,w$ have been fixed and clear from the context) the parallelogram with sides orthogonal to $v$ and $w$, whose barycenter is $0$, and with sides of length $1$. We then denote $P_{v,w}(x,r):=x+rP_{v,w}$, and we remove the dependence from $v,w$ if they are fixed from the context. 
 \end{definizione}


\begin{definizione}[Touching parallelograms]
\label{def:touching_parallelogram}
    Given a closed set $E\subset\R^2$ and a parallelogram $P=P_{v,w}$, for every $x\in \R^2\setminus E$ we consider 
    \[
    d_P(x,E):=\sup\{r>0:\, P(x,r)\cap E=\emptyset\}
    \]
    and we call $P(x):=P(x,d_P(x,E))$ the \textit{touching parallelogram centered at $x$}.
\end{definizione}


In the final proof of Marstrand's theorem we will apply the following proposition with $E=\supp(\mu)$.

\begin{proposizione}\label{prop:touching_at_vertices}
    Fix a parallelogram $P=P_{v,w}$. Suppose that $E\subsetneq \R^2$ is a closed set, for which every touching parallelogram intersects $E$ only at the vertices. Then there exists a point $x_0\in E$, a radius $r>0$, a line $\ell$ passing through $x_0$ and a Lipschitz function $f:\R\to \R$ with $f(0)=0$ such that 
    \[
    E\cap B(x_0,r)\subseteq \{x\in \R^2: \, \langle x, v_{\ell^\perp}\rangle \leq f(\pi_\ell(x)) \}\qquad\text{and}\qquad\mathrm{gr}(f)\cap B(x_0,r)\subseteq \supp(\mu),
    \]
    where $v_{\ell^\perp}$ is a unitary vector in $\ell^\perp$.
\end{proposizione}

\begin{figure}
    \centering
    \def\svgscale{0.7}{
\begingroup%
  \makeatletter%
  \providecommand\color[2][]{%
    \errmessage{(Inkscape) Color is used for the text in Inkscape, but the package 'color.sty' is not loaded}%
    \renewcommand\color[2][]{}%
  }%
  \providecommand\transparent[1]{%
    \errmessage{(Inkscape) Transparency is used (non-zero) for the text in Inkscape, but the package 'transparent.sty' is not loaded}%
    \renewcommand\transparent[1]{}%
  }%
  \providecommand\rotatebox[2]{#2}%
  \newcommand*\fsize{\dimexpr\f@size pt\relax}%
  \newcommand*\lineheight[1]{\fontsize{\fsize}{#1\fsize}\selectfont}%
  \ifx\svgwidth\undefined%
    \setlength{\unitlength}{202.79881515bp}%
    \ifx\svgscale\undefined%
      \relax%
    \else%
      \setlength{\unitlength}{\unitlength * \real{\svgscale}}%
    \fi%
  \else%
    \setlength{\unitlength}{\svgwidth}%
  \fi%
  \global\let\svgwidth\undefined%
  \global\let\svgscale\undefined%
  \makeatother%
  \begin{picture}(1,0.88635337)%
    \lineheight{1}%
    \setlength\tabcolsep{0pt}%
    \put(0,0){\includegraphics[width=\unitlength,page=1]{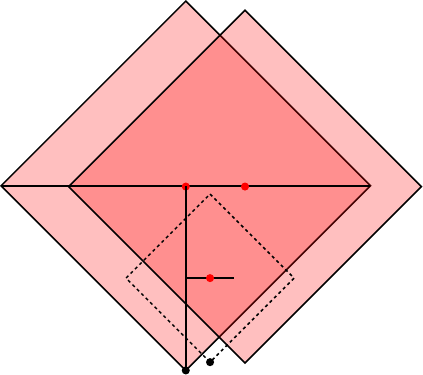}}%
    \put(0.58644349,0.49730681){\color[rgb]{0,0,0}\makebox(0,0)[lt]{\lineheight{1.25}\smash{\begin{tabular}[t]{l}$x_\eps$\end{tabular}}}}%
    \put(0.49118083,0.27689497){\color[rgb]{0,0,0}\makebox(0,0)[lt]{\lineheight{1.25}\smash{\begin{tabular}[t]{l}$y_t$\end{tabular}}}}%
  \end{picture}%
\endgroup%
}
    \caption{Reference figure for the proof of Proposition \ref{prop:touching_at_vertices}.}
    \label{fig:touching_at_vertices}
\end{figure}

\begin{proof}
Let us denote by $P(x)$ the maximal parallelogram centered at $x$. Up to affine transformations, that do not affect the problem, we can assume that $P$ is a square with sidelength $\sqrt{2}$ and sides parallel to the bisectors of the quadrants, that $0\in E$, that $e_2=(0,1)\not\in E$ and that $\partial P(e_2)\cap E=\{0\}$. The latter property can always be ensured just by considering a possibly smaller square and by rescaling. Consider now, for some small $\eps>0$, the point $x_\eps:=e_2+\eps e_1$ and the maximal parallelogram $P(x_\eps)$. It is easy to check that $P(x_\eps)=P(x_\eps,s)$ for some $1-\eps < s< 1+\eps$. 

Consider now the points $y_t:=(t,\tfrac12)$ and the corresponding parallelograms $P(y_t)$. We claim that, for some $\delta>0$, it holds that for $0<t<\delta$ the set $\partial P(y_t)\cap E$ consists of the bottom vertex only. Indeed, all other vertices are entirely contained in the interior of $P(e_2)\cup P(x_\eps)$ (pictured in red in Figure \ref{fig:touching_at_vertices}) and otherwise $E$ would intersect $P(y_t)$ (pictured with a dotted boundary) at an interior point of an edge, which is not allowed.

It follows that for every $0<t<\eps$ there is a point $(t,f(t))$ contained in $\supp(\mu)$, where $f$ is a 1-Lipschitz function, and that locally $E\subset \{(t,s):\,s\le f(t)\}$.
\end{proof}

\begin{proposizione}\label{prop:flat_support}
    Let $\lVert\cdot\rVert$ be a norm in $\R^n$ such that $e_n$ is a direction of weak monotonicity. Let moreover $\mu$ be an $\alpha$-uniform measure in $\R^n$ such that $\supp(\mu)\subseteq \{x_n\ge 0\}$ and $\{x_n=0\}\subseteq \supp(\mu)$. Then every point $z_0\in\{x_n=0\}$ admits a strong tangent cone, namely there exist $M>0$ and  $r=r(z_0)>0$ such that
    \[
    \supp(\mu)\cap B(z_0,r)\subset X(z_0,\{x_n=0\},M).
    \]
    In particular, there exist $z_0\in\{x_n=0\}$ and $\nu\in\Tan_\alpha(\mu,z_0)$ such that $\supp(\nu)\subseteq \{x_n=0\}$.
\end{proposizione}

\begin{proof}
    Let us set for simplicity $V:=\{x_n=0\}$. Consider any $z_0\in V$ and any $\nu\in\Tan_\alpha(\mu,z_0)$. It is clear by Lemma \ref{replica} that we still have
    \[
    V\subseteq \supp(\nu).
    \]
    If $\supp(\nu)=V$ then we are done. Otherwise it means that $\mathrm{span}(\supp(\nu))=\R^n$, hence by Corollary \ref{cor:barycenter_zero} we deduce that $b_\nu(\rho)=0$ for every $\rho>0$. By Lemma \ref{lemma:monotonicity_and_flatness}, Point 2, we conclude that $\supp(\nu)\subseteq X(0,V,M)$ for some $M>0$. Now taking into account the $\alpha$-uniformity of $\mu$, this implies that there exists $r(z_0)>0$ such that 
    \begin{equation}\label{eq:strong_cone}
        \supp(\mu)\cap B(z_0,r)\subset X(z_0,V,2M)\qquad\text{for every $0<r\le r(z_0)$.}
    \end{equation}
    Indeed, if this were not the case we could find a tangent measure $\nu\in\Tan_\alpha(\mu,z_0)$ whose support is not contained in $X(0,V,M)$, contradiction.

    Let now, for $j\in\mathbb{N}^+$,
    \[
    E_j:=\{x\in V:\, r(x)\ge j^{-1}\}.
    \]
    Then $V=\bigcup_j E_j$, each $E_j$ is closed (and thus $\mu$-measurable) and therefore
    by sigma-additivity there must exist $j_0\in\mathbb{N}$ such that $\H^{n-1}(E_{j_0})>0$. Let us consider a point $z_0$ of density 1 for the set $E_{j_0}$ with respect to the measure $\H^{n-1}\llcorner V$. We claim that any tangent $\nu\in\Tan_\alpha(\mu,z_0)$ satisfies $\supp(\nu)\subseteq \{x_n=0\}$. This directly follows from the following fact: for every $\delta>0$ there is $r(\delta)>0$ such that $\supp(\mu)\cap B(z_0,r)\subseteq I_{\delta r}(V)$ for every $0<r<r(\delta)$, where $I_\eps(V):=\{x:\,\dist(x,V)<\eps\}$ denotes the $\eps$-neighborhood of $V$.

    Let us prove the latter fact. From the density property of $z_0$ for every $R>0$ and every $\delta>0$ there exists a $\rho_0>0$ such that for every $r\leq \rho_0$ we have 
    $$\frac{\mathcal H^{n-1} \llcorner V(B(z_0,Rr)\cap E_j)}{ \mathcal H^{n-1}\llcorner V(B(0,1)) R^{n-1}r^{n-1}}\geq 1-R^{-(n-1)}\delta^{n-1}.$$
    Thanks to this choice, we know that in the ball $B(z_0,Rr)$, the set $E_j\cap B(z_0,Rr)$ must be $\delta  r$-dense (namely, its $\delta r$-neighborhood contains the ball $B(z_0,Rr)$). This implies in particular, thanks to \eqref{eq:strong_cone}, that
    $$\supp\mu\cap B(z_0,R r)\subseteq B(V,\delta r),$$
    which in turn implies that 
    $$\supp(r^{-\alpha}T_{z_0,r}\mu)\cap B(0,R)\subseteq B(0,R)\cap B(V,\delta ).$$
    Thanks to the arbitrariness of $R$ and of $\delta$, this concludes the proof of our claim and of the proof.
\end{proof}

\section{Proof of Marstrand's theorem}\label{sec:proof}

We finally put together all the ingredients to obtain the full proof of Theorem \ref{thm:main}.

\begin{enumerate}
    \item[(i)] \textbf{Setup.} Let us assume by contradiction that there exists $\alpha\in \mathbb{R}\setminus\mathbb{N}$ and a non-trivial measure $\mu$ in $\R^2$, equipped with a norm $\lVert\cdot\rVert$, such that 
    \[
    \lim_{r\to 0}\frac{\mu(B(x,r))}{r^\alpha}
    \]
    exists positive and finite at $\mu$-a.e. point. By Proposition \ref{prop:tangents_are_uniform}, by taking a tangent measure at a typical point we can assume that $\mu$ is $\alpha$-uniform. By Remark \ref{rmk:not_everything} and Remark \ref{rmk:uniform<1} we can restrict to the case $\alpha\in(1,2)$. Our goal will be showing that, by taking tangent measures at specific points, we can find a non-trivial $\alpha$-uniform measure that is supported on a line, which is a contradiction again by Remark \ref{rmk:not_everything}. Let us recall that by Proposition \ref{uniformup}, tangents at \emph{every  point } of the support of an $\alpha$-uniform measure exist and are $\alpha$-uniform measures.
    
    
    \item[(ii)] \textbf{Differentiability of the norm.} In the following we can assume without loss of generality that the norm $\lVert\cdot\rVert$ is differentiable at $\mu$-a.e. point, for every measure $\mu$ that we will consider (and thus that the nonlinear barycenter \eqref{eq:nonlinear_barycenter} is well-defined). Indeed, by Lemma \ref{lemma:non-differentiability} we know that the set $N$ of non-differentiability of $\lVert\cdot\rVert$ is contained in a countable union of lines through the origin. If $\mu(N)>0$ then at least one of these lines, call it $\ell$, satisfies $\mu(\ell)>0$. Thus by taking a tangent at a density point of $\mu\llcorner \ell$ we would obtain an $\alpha$-uniform measure supported on a line, and this would be a contradiction.
    
    \item[(iii)] \textbf{Touching point argument.} By Remark \ref{rmk:not_everything} the support of $\mu$ is a closed set whose complement is non-empty. By Lemma \ref{lemma:directions_of_monotonicity} we can find a linear change of variables $A:\R^2\to\R^2$ such that the norm $\lVert\cdot\rVert_A$ admits two independent directions of strict monotonicity $v,w$. Let us consider parallelograms whose sides are alternatively orthogonal to $v$ and $w$, and all of the same length. 
    Let us consider for every point $x\in\R^2\setminus\supp(\mu)$ the touching parallelogram centered at $x$. There are two cases:
    \begin{enumerate}
        \item either there exists a point $x\in \R^2\setminus \supp(\mu)$ such that the touching parallelogram centered at $x$ intersects $\supp(\mu)$ in at least one non-vertex point $z$;
        \item or for every point $x\in\R^2\setminus \supp(\mu)$ the touching parallelogram centered at $x$ intersects $\supp(\mu)$ only at the vertices.
    \end{enumerate}

    \begin{enumerate}
    
    \item[(iii.a)] 
    Let us take a tangent measure $\nu\in \Tan_\alpha(\mu,z)$. Up to an isometry, and without loss of generality, we can assume that $v=e_2$ and that $z$ belongs to the side orthogonal to $e_2$, and thus that $\supp(\nu)\subseteq \{x_2\ge 0\}$ and $0\in\supp(\nu)$. Let us consider a further tangent measure $\eta\in\Tan_\alpha(\nu,0)$. The span of $\supp(\eta)$ must be the whole plane (otherwise $\eta$ would be an $\alpha$-uniform measure supported on a line, $\alpha\in(1,2)$, a contradiction). It follows from Corollary \ref{cor:barycenter_zero} that $b_\eta(\rho)=0$ for every $\rho> 0$. By the strict monotonicity of the direction $e_2$, from Lemma \ref{lemma:monotonicity_and_flatness} it descends that $\supp(\eta)\subseteq \{x_2=0\}$, hence we find a contradiction anyway.
    
    \item[(iii.b)] 
    In this case by Proposition \ref{prop:touching_at_vertices} we can find a point $x_0\in \supp(\mu)$ and a radius $r>0$ such that $\supp(\mu)\cap B(x_0,r)$ is contained in the subgraph of a Lipschitz function $f$ (with respect to suitable axes), and moreover such that the graph of $f$, $\mathrm{gr}(f)$, satisfies $\mathrm{gr}(f)\cap B(x_0,r)\subseteq \supp(\mu)$. We now take a tangent measure at a point $z_0\in \mathrm{gr}(f)\cap B(x_0,r)$ that corresponds to a differentiability point of $f$, which exist by Rademacher's theorem. We end up with an $\alpha$-uniform measure $\nu\in \Tan_\alpha(\mu,z_0)$ that, up to rigid motions, satisfies the following:
    \begin{equation}\label{eq:support_on_one_side}
    \supp(\nu)\subseteq \{x_2\ge 0\},\qquad \{x_2=0\}\subseteq \supp(\nu).
    \end{equation}
    Observe now that $e_2$ is not necessarily a direction of weak monotonicity for $\norm$; however by {Lemma \ref{lemma:monotonicity_linear}} we can consider a linear transformation $A:\R^2\to \R^2$ that preserves $\{x_2=0\}$ and for which $e_2$ becomes a direction of weak monotonicity for $\norm_A$. The measure $\nu_A:=A_\# \nu$ is $\alpha$-uniform with respect to the norm $\lVert \cdot\rVert_A$  and still satisfies 
    \eqref{eq:support_on_one_side}. By taking a further tangent measure we end up with $\eta\in \Tan_\alpha(\nu,0)$ that also satisfies $b_\eta(\rho)=0$ for every $\rho>0$ by Corollary \ref{cor:barycenter_zero} (again, if the span of $\supp(\eta)$ is not $\R^2$ we are done anyway). By Proposition \ref{prop:flat_support} we can find yet another tangent measure $\theta\in\Tan_\alpha(\eta,z_0)$, for some $z_0\in \{x_2=0\}$, such that $\supp(\theta)\subseteq \{x_2=0\}$. This is again a contradiction, hence also this case is not possible. The proof is concluded.
    \end{enumerate}
\end{enumerate}

\printbibliography

\end{document}